\documentclass[12pt, reqno]{amsart}

\newcommand{\zet}{\mathbb{Z}}

\usepackage{amsmath, amsthm, amssymb, stmaryrd}
\usepackage{hyperref}
\usepackage{enumerate}
\usepackage{url}

\usepackage{color}

\usepackage[centering, margin={1in, 0.5in}, includeheadfoot]{geometry}

\usepackage{fancyvrb}

\newtheorem{thm}{Theorem}[section]
\newtheorem{lemma}[thm]{Lemma}
\newtheorem{prop}[thm]{Proposition}

\newtheorem{conj}[thm]{Conjecture}

\theoremstyle{definition}

\theoremstyle{remark}

\theoremstyle{task}

\theoremstyle{qn}

\numberwithin{equation}{section}

\newcommand{\mmod}[1]{\,\,\text{mod}\,\,#1}

\def\alp{{\alpha}} 
\def\bet{{\beta}}  
\def\gam{{\gamma}} 
\def\del{{\delta}} \def\Del{{\Delta}}

\def\kap{{\kappa}}
\def\lam{{\lambda}}

\def\eps{\varepsilon}

\def\le{\leqslant} \def\ge{\geqslant}

\def \bA {\mathbb A}

\def \bC {\mathbb C}

\def \bN {\mathbb N}

\def \bQ {\mathbb Q}
\def \bR {\mathbb R}
\def \bZ {\mathbb Z}

\def \bzero {\mathbf 0}

\def \cF {\mathcal F}

\def \cJ {\mathcal J}

\def \cL {\mathcal L}

\def \cP {\mathcal P}

\def \cR {\mathcal R}
\def \cS {\mathcal S}

\def \cU {\mathcal U}
\def \cV {\mathcal V}

\def \deg {\mathrm{deg}}

\def \disc{\mathrm{disc}}

\begin{document}
\title[Enumerative Galois theory for cubics and quartics]{Enumerative Galois theory for cubics and quartics}
\subjclass[2010]{11R32 (primary); 11C08, 11G35 (secondary)}
\keywords{Galois theory, determinant method, Mahler measure}
\author{Sam Chow \and Rainer Dietmann}
\address{Mathematics Institute, Zeeman Building, University of Warwick, Coventry CV4 7AL, United Kingdom}
\email{Sam.Chow@warwick.ac.uk}
\address{Department of Mathematics, Royal Holloway, University of London\\
Egham TW20 0EX, United Kingdom}
\email{Rainer.Dietmann@rhul.ac.uk}

\begin{abstract}
We show that there are $O_\varepsilon(H^{1.5+\varepsilon})$ monic, cubic polynomials with integer coefficients bounded by $H$ in absolute value whose Galois group is $A_3$. We also show that the order of magnitude for $D_4$ quartics is $H^2 (\log H)^2$, and that the respective counts for $A_4$, $V_4$, $C_4$ are $O(H^{2.91})$, $O(H^2 \log H)$, $O(H^2 \log H)$. Our work establishes that irreducible non-$S_3$ cubic polynomials are less numerous than reducible ones, and similarly in the quartic setting: these are the first two solved cases of a 1936 conjecture made by van der Waerden.
\end{abstract}
\maketitle

\section{Introduction}

Consider monic polynomials 
\begin{equation} \label{fdef}
f(X) = X^n + a_1 X^{n-1} + \cdots + a_{n-1} X + a_n
\end{equation}
of a given degree $n \ge 3$, with integer coefficients. Recall that the \emph{Galois group} $G_f$ of $f$ is the automorphism group of its splitting field. As $G_f$ acts on the roots of $f$, it can be embedded into the symmetric group $S_n$; the only information that we will need about inseparable polynomials is that their Galois group is not isomorphic to $S_n$. The enumeration of polynomials with prescribed Galois group is an enduring topic. 

\subsection{Van der Waerden's conjecture}

Van der Waerden \cite{vdW1936} showed that a generic polynomial has full Galois group, and a popular objective has been to sharpen his bound on the size
\[
E_n(H) := \# \{ (a_1, \ldots, a_n) \in (\bZ \cap [-H,H])^n: G_f \not \simeq S_n \}
\]
of the exceptional set, where $H$ is a large positive real number. Van der Waerden obtained
\[
E_n(H) \ll_n H^{n- \: \frac1{6(n-2) \log \log H}}, 
\]
and made the following conjecture. We write $R_n(H)$ for the number of monic, reducible polynomials of degree $n$, with integer coefficients in $[-H,H]$.

\begin{conj} [van der Waerden, 1936]
For $n \ge 3$, we have
\[
E_n(H) = R_n(H) (1 + o(1)).
\] 
\end{conj}

\noindent We have paraphrased slightly: van der Waerden suggested that monic, irreducible, non-$S_n$ polynomials of degree $n$ are rarer than monic reducibles, counted in this way. It follows from the proof in \cite{Che1963} that if $n \ge 3$ then
\begin{equation} \label{ReducibleCount}
R_n(H) = c_n H^{n-1} + O_n( H^{n-2} (\log H)^2),
\end{equation}
for some constant $c_n > 0$. Chela \cite{Che1963} stated this without an explicit error term, and in Appendix \ref{Red} we explain how the error term in \eqref{ReducibleCount} comes about. Van der Waerden's conjecture may therefore be equivalently stated as follows.

\begin{conj} For $n \ge 3$, the number of monic, irreducible, non-$S_n$ polynomials of degree $n$, with coefficients in $\bZ \cap [-H,H]$, is $o(H^{n-1})$ as $H \to \infty$.
\end{conj}

Hitherto, no case of this conjecture was known. In the cubic case $n=3$, Lefton \cite{Lef1979} showed that $E_3(H) \ll_\eps H^{2+\eps}$, a record that has stood unbeaten for over four decades. We establish the following asymptotic formula for $E_3(H)$, thereby resolving the cubic case of van der Waerden's conjecture.

\begin{thm} \label{CubicVDW} For any $\eps > 0$ we have
\[
E_3(H) = 8 \Bigl(\frac{\pi^2}6 + \frac14\Bigr) H^2 + O_\eps(H^{1.5 + \eps}).
\]
\end{thm}

\noindent Note from \cite{Che1963} that $c_3 = 8 (\frac{\pi^2}6 + \frac14) $, so we draw the following equivalent conclusion.

\begin{thm} \label{nonS3}
The number of monic, irreducible, non-$S_3$ cubic polynomials 
\begin{equation} \label{CubicPolynomial}
f(X) = X^3 + aX^2 + bX + c
\end{equation}
with $a,b,c \in \bZ \cap [-H,H]$ is $O_\eps(H^{1.5 + \eps})$.
\end{thm}

It was thought that the second author \cite{Die2006} had come close to settling the quartic case $n=4$ over a decade ago, asserting the estimate
\begin{equation} \label{ConjBound}
E_4(H) \ll_\eps H^{3+\eps}.
\end{equation}
However, we have discovered an error in Eq. (7) therein, which appears to damage the argument beyond repair---see \cite[p. 613]{DF2004} for the correct expressions. To our knowledge, the strongest unconditional bound to date is $E_4(H) \ll_\eps H^{2+ \sqrt2 + \eps}$, obtained in \cite{Die2013}. The inequality \eqref{ConjBound} is known conditionally \cite[Theorem 1.4]{Won2005}.

We establish the following asymptotic formula for $E_4(H)$, thereby settling the quartic case of van der Waerden's conjecture.

\begin{thm} \label{asymptotic} For any $\eps > 0$ we have
\[
E_4(H) = 16 \Bigl(\zeta(3) + \frac16 \Bigr) H^3 + O_\eps(H^{\frac52+ \frac1{\sqrt 6} + \eps}).
\]
\end{thm}

\noindent Note that $\frac52+ \frac1{\sqrt 6} \le 2.91$, and note from \cite{Che1963} that $c_4 = 16(\zeta(3) + \frac16)$, so if only irreducible polynomials are considered then the exponent is lower than 3.

\begin{thm} \label{nonS4}
The number of monic, irreducible, non-$S_4$ quartic polynomials 
\begin{equation} \label{consider}
f(X) = X^4 + aX^3 + bX^2 + cX + d
\end{equation}
with $a,b,c,d \in \bZ \cap [-H,H]$ is $O_\eps(H^{\frac52+\frac1{\sqrt 6} + \eps})$.
\end{thm}

\noindent Theorem \ref{nonS4} shows that irreducible non-$S_4$ quartics are less numerous than reducible quartics, and is equivalent to Theorem \ref{asymptotic}.

\subsection{Specific groups}

We now address the general problem of counting polynomials with prescribed Galois group. For $G \le S_n$, let us write $N_{G,n} = N_{G,n}(H)$ for the number of monic, irreducible, integer polynomials, with coefficients bounded by $H$ in absolute value, whose Galois group is isomorphic to $G$. The second author showed in \cite{Die2012} that
\begin{equation} \label{general}
N_{G,n} \ll_{n,\eps} H^{n-1 + \frac1{[S_n : G]} + \eps},
\end{equation}
and in \cite{Die2013} that
\[
N_{A_n, n} \ll_{n,\eps} H^{n-2+\sqrt 2 + \eps}.
\]
The latter article established that if $n \ge 3$ then
\[
E_n(H) \ll_{n,\eps} H^{n-2+\sqrt 2 + \eps},
\]
breaking a record previously held by van der Waerden \cite{vdW1936}, Knobloch \cite{Kno1956}, Gallagher \cite{Gal1972} and Zywina \cite{Zyw2010}. 

Recall that if $f$ is irreducible then $G_f$ acts transitively on the roots of $f$. Thus, in the cubic case $n=3$, the only possibilities for the Galois group of an irreducible cubic polynomial are $S_3$ and $A_3$. The polynomials counted in Theorem \ref{nonS3} are the $A_3$ cubics, and the others are either reducible or have full Galois group. Our bound $N_{A_3,3} \ll_\eps H^{1.5+\eps}$ dramatically improves upon Lefton's longstanding record of $N_{A_3,3} \ll_\eps H^{2+\eps}$. 

Using the C programming language, we found that 
\[
N_{A_3,3}(2000) = 355334
\]
(for the code, see Appendix \ref{Code}). From the additional data point $N_{A_3,3}(500) = 52420$, one might empirically estimate the exponent as $\log(355334/52420)/\log4 \approx 1.38$. The best lower bound that we know of is
\[
N_{A_3,3}(H) \gg H,
\]
coming from the one-parameter family $X^3 + tX^2 + (t - 3)X - 1$ given for example in Smith's tables \cite[\S 12]{Smi2000}. So the correct exponent, if well-defined, lies between 1 and 1.5.\\

Now consider the quartic case $n=4$. In this case there are five possibilities for $G_f$, namely $S_4$, $A_4$, $D_4$, $V_4$ and $C_4$, see \cite{KW1989}. Here $D_4$ is the dihedral group of order 8, and $A_4, V_4$ are respectively the alternating and Klein four groups. As usual $C_4$ is the cyclic group of order $4$. We write $\cS_H$ for the set of monic, irreducible quartics with coefficients in $\bZ \cap [-H,H]$, and for $G \in \{ S_4, A_4, D_4, V_4, C_4 \}$ we define
\[
N_G = N_G(H) = \# \{ f \in \cS_H: G_f \simeq G \}.
\]

We ascertain the order of magnitude for the number of $D_4$ quartics. To our knowledge, this is the first time that the order of magnitude of $N_{G,n}$ has been obtained, for $G \not \simeq S_n$.

\begin{thm} \label{thm1}
We have
\[
N_{D_4} \asymp H^2 ( \log H)^2.
\]
\end{thm}

\noindent In addition, we show that $V_4$ and $C_4$ quartics are less numerous.

\begin{thm} \label{thm2}
We have
\[
N_{V_4} + N_{C_4} \ll H^2 \log H.
\]
\end{thm}

\noindent Finally, to complete the proof of Theorem \ref{nonS4}, we establish the following upper bound for $A_4$ quartics.

\begin{thm} \label{thm3}
We have
\[
N_{A_4} \ll_\eps H^{\frac52+\frac1{\sqrt{6}}+\eps}.
\]
\end{thm}

We searched the literature for constructions that imply lower bounds for these quantities.
Working from \cite{NV1983}, one obtains \mbox{$N_{A_4} \gg H$},
see \S \ref{sonne}.
We can deduce from \cite[\S 12]{Smi2000} and \cite[Theorem 2.1]{Coh1981} that \mbox{$N_{C_4} \gg H$;} the latter cited result is based on a quantitative version of Hilbert's irreducibility theorem. We can construct a family of quartics that implies a sharper lower bound for $N_{V_4}$ than what we were able to find in the literature: the construction given in \S \ref{lower} shows that \mbox{$N_{V_4} \gg H^{3/2}$.}

We summarise our state of knowledge concerning the quartic case as follows:
\begin{align*}
N_{S_4} &= 16H^4 + O(H^3)\\
N_{D_4} &\asymp H^2 (\log H)^2 \\
H^{3/2}  \ll N_{V_4} &\ll H^2 \log H \\
H \ll N_{C_4} &\ll H^2 \log H \\
H \ll N_{A_4} &\ll_\eps  H^{\frac52+\frac1{\sqrt 6} + \eps}.
\end{align*}
The story is still far from complete. We expect that in time asymptotic formulas will emerge for every $N_{G,4}(H)$. Below we provide the values of $N_{G,4}(150)$, evaluated using the C programming language (for the code, see Appendix \ref{Code}).

\begin{center}
\begin{tabular}{|c|c|c|}
\hline
$G$ & $N_{G,4}(150)$ \\
\hline
$S_4$ & 8128593894 \\
$A_4$ & 60954 \\
$D_4$ & 4501148 \\
$V_4$ & 45953 \\
$C_4$ & 11818 \\
\hline
$f$ is reducible & 75327434 \\
\hline
\end{tabular}
\end{center}

\vspace{3mm}
\noindent
This suggests that the upper bounds for $A_4, V_4$ and $C_4$ quartics may be far from the truth.

We remark that our counting problem differs substantially from the corresponding problem for quartic fields, for which Bhargava \cite{Bha2005} showed that in some sense a positive proportion of quartic fields have Galois group $D_4$. For an explanation of why the results are consistent, see \cite[Remark 5.1]{Won2005}. 

\subsection{Parametrisation, concentration, and root separation}

Cubics \eqref{CubicPolynomial} with Galois group $A_3$ have non-zero square discriminant $(4I^3 - J^2) / 27$, where
\begin{equation} \label{IJdefCubic}
I = a^2 - 3b, \qquad J = 27c - 9ab + 2a^3.
\end{equation}
This leads us to the diophantine equation
\begin{equation} \label{dihedral1}
J^2 + 3Y^2 = 4I^3,
\end{equation}
and we can parametrise the solutions using algebraic number theory. This equation is discussed in \cite[\S 14.2.3]{Coh2007} and elsewhere \cite{DG1995}, but here we also need to deal with common divisors between the variables, and these can be enormous. Accounting for the common divisors gives rise to parametrised families of $(I,J,Y)$ encompassing all solutions to the diophantine equation \eqref{dihedral1}. The broad idea is to count those pairs $(I,J)$ with the parameters lying in given dyadic ranges, and then to count possibilities for the corresponding $a,b,c$ subject to those ranges. 

To illustrate the concentration method, consider the discriminant. On one hand, this is $O(H^4)$, being quartic in $a,b,c$. On the other hand, based on \eqref{IJdefCubic}, we would expect it to have size roughly $H^6$. For concreteness, one of the parametrised families of solutions to \eqref{dihedral1} is
\[
(J, Y, I) = (2s^3 - 18st^2, 6t (s-t)(s+t), s^2 + 3t^2),
\]
where $s,t \ll H$. Now $t (t-s) (t+s) = -Y/6 = -\sqrt \Del/2 \ll H^2$, imposing a constraint on $s,t$. Writing $\lam = t/s$, one interpretation is that if $s$ is not small then $\lam (\lam - 1) (\lam + 1)$ is small, so the ratio $t/s$ must be close to a root of the polynomial $X(X-1)(X+1)$. In other words, either $s \approx 0$ or $t \approx 0$ or $s \approx t$ or $s \approx -t$. This restriction on the pair $(s,t)$ delivers a saving.

Four instances of concentration arise in our proof. In the first, the concentrating polynomials are linear, and the rewards are easily harvested. In the second, the concentrating polynomials are cubic, and the roots are well-separated, owing to (i) Mahler's work \cite{Mah1964} involving what is now known as the Mahler measure \cite{Smy2008}, and (ii) the discriminant always being bounded well away from zero. In the third, the concentrating polynomials are quadratic, and we can consider a difference of perfect squares. In the final instance, the concentrating polynomials are cubic, but are ``close'' to being quadratic, and we can again consider a difference of perfect squares.

\subsection{New and old identities}

Our investigation of the quartic case begins with classical criteria \cite{KW1989} involving the \emph{discriminant} and \emph{cubic resolvent} of a monic, irreducible quartic polynomial \eqref{consider}. When the Galois group is $D_4$, $V_4$ or $C_4$, the cubic resolvent has an integer root, which we introduce as an extra variable $x$. Changing variables to use $e = b-x$ instead of $b$, we obtain the astonishing symmetry \eqref{symmetry}, which we believe is new. For emphasis, the identity is
\[
(x^2 - 4d) \cdot (a^2 - 4e) = (xa-2c)^2.
\]
Using ideas from the geometry of numbers and diophantine approximation leads to the upper bound
\begin{equation} \label{GroupUpper}
N_{D_4} + N_{V_4} + N_{C_4} \ll H^2 (\log H)^2.
\end{equation}
The proof then motivates a construction that implies the matching lower bound
\begin{equation} \label{LowerBound}
N_{D_4} + N_{V_4} + N_{C_4} \gg H^2 (\log H)^2.
\end{equation}

The analysis described above roughly speaking provides an approximate parametrisation of the $D_4$, $V_4$ and $C_4$ quartics, by certain variables $u,v,w,x,a$, where $a$ is as in \eqref{consider}. To show that $N_{V_4}$ and $N_{C_4}$ satisfy the stronger upper bound $O(H^2\log H)$, we use an additional piece of information in each case; this takes the form of an equation $y^2 = P_{u,v,w,a}(x)$, where $P_{u,v,w,a}$ is a polynomial and $y$ is an additional variable. We require upper bounds for the number of integer solutions to this diophantine equation in $(x,y)$, and these bounds need to be uniform in the coefficients. We are able to ascertain that the curve defined is absolutely irreducible, which enables us to apply a Bombieri--Pila \cite{BP1989} style of result by Vaughan \cite[Theorem 1.1]{Vau2014}.\\

Our study of $A_4$ quartics starts with the standard fact that the discriminant is in this case a square. Deviating from previous work on this topic, we employ the invariant theory of $\mathrm{GL}_2$ actions on binary quartic forms (or, equivalently, unary quartic polynomials), see \cite{BS2015}. The discriminant can then be written as $(4I^3-J^2)/27$, where
\begin{equation} \label{IJdef}
I = 12d-3ac+b^2, \qquad J = 72bd + 9abc - 27c^2 - 27a^2d - 2b^3.
\end{equation}
Our strategy is first to count integer solutions $(I,J,y)$ to
\begin{equation} \label{IJeq}
4I^3 - J^2 = 27 y^2,
\end{equation}
and then to count integer solutions $(a,b,c,d)$ to the system \eqref{IJdef}. In the latter step, we require upper bounds that are uniform in the coefficients. Further manipulations lead us to an affine surface $Y_{I,J}$, which we show to be absolutely irreducible. A result stated by Browning \cite[Lemma 1]{Bro2011}, which he attributes to Heath-Brown and Salberger, then enables us to cover the integer points on the surface by a reasonably small family of curves. By showing that $Y_{I,J}$ contains no lines, and using this fact nontrivially, we can then decompose each curve in the family into irreducible curves of degree greater than or equal to 2, and finally apply Bombieri--Pila \cite{BP1989}. 

For convenient reference, we record below a version of the Kappe--Warren criterion \cite{KW1989}, as given in an expository note of Keith Conrad's \cite[Corollary 4.3]{Con}. The distinction between $D_4$ and $C_4$ is done slightly differently between those two documents; Conrad's description of this is readily deduced from \cite[Theorem 13.1.1]{Cox} and the identity \eqref{MainEq}. We will see in \S \ref{res} that the cubic resolvent of a monic, quartic polynomial with integer coefficients is a monic, cubic polynomial with integer coefficients. Also note that if $f(X) \in \bZ[X]$ is irreducible then its discriminant $\Del$ is a non-zero integer. 

\begin{thm} [Kappe--Warren criterion] \label{KW} For a monic, irreducible quartic $f(X) \in \bZ[X]$, whose cubic resolvent is $r(X)$, the isomorphism class of the Galois group $G_f$ is as follows.

\begin{center}
\begin{tabular}{|c|c|c|c|} 
\hline
&&&\\[-1em]
$\Del \in \bZ$ & $r(X) \in \bZ[X]$ & $(x^2 - 4d)\Del, (a^2 - 4(b-x))\Del \in \bZ$ & $G_f $\\
\hline
&&&\\[-1em]
$\ne \square$ & irreducible && $S_4$ \\
$=\square$ & irreducible && $A_4$ \\
$\ne \square$ & unique root $x \in \bZ$ & at least one $\ne \square$ & $D_4$ \\
$\ne \square$ & unique root $x \in \bZ$ & both $=\square$ & $C_4$  \\
$=\square$ & reducible && $V_4$ \\
\hline
\end{tabular} 
\end{center}
\end{thm}

\medskip

\subsection*{Organisation} The cubic case is handled in \S \ref{CubicCase}.  In \S \ref{res} we establish \eqref{GroupUpper}, and in \S \ref{cons} we prove the complementary lower bound \eqref{LowerBound}. In \S \ref{V4C4}, we establish Theorem \ref{thm2}, thereby also completing the proof of Theorem \ref{thm1}. In \S \ref{A4} we prove Theorem \ref{thm3}, thereby also completing the proof of Theorem \ref{nonS4}. Finally, in \S \ref{LowerBounds}, we show that $N_{V_4} \gg H^{3/2}$ and $N_{A_4} \gg H$. Appendix \ref{Code} contains the C code used to compute the values of $N_{G,4}(150)$, for $G \in \{ S_4, A_4, D_4, V_4, C_4 \}$, and also the code used to compute $N_{A_3,3}(2000)$. In Appendix \ref{Red}, we verify the error term in \eqref{ReducibleCount}. In Appendix \ref{Aux}, we show that if the discriminant $4I^3 - J^2$ is non-zero then the set of binary forms with given invariants $I$ and $J$ contains no rational lines; this is related to Lemma \ref{NoLines} and is of independent interest.

\subsection*{Notation} We adopt the convention that $\eps$ denotes an arbitrarily small positive constant, whose value is allowed to change between occurrences. We use Vinogradov and Bachmann--Landau notation throughout, with the implicit constants being allowed to depend on $\eps$. We write $\#S$ for the cardinality of a set $S$. If $g$ and $h$ are positive-valued, we write $g \asymp h$ if $g \ll h \ll g$. Throughout $H$ denotes a positive real number, sufficiently large in terms of $\eps$. Let $\mu(\cdot)$ be the M\"obius function.

\subsection*{Funding and acknowledgments} The first author gratefully acknowledges the support of EPSRC Fellowship Grant EP/S00226X/1, EPSRC Fellowship Grant EP/S00226X/2, EPSRC Programme Grant EP/J018260/1, an Oberwolfach Leibniz Graduate Students grant, and the National Science Foundation under Grant No. DMS-1440140 while in residence at the Mathematical Sciences Research Institute in Berkeley, California, during the Spring 2017 semester. Both authors thank the Mathematisches Forschungsinstitut Oberwolfach and the Fields Institute for excellent working conditions, and the second author would like to thank the Mathematical Institute at the University of Oxford for hosting him during a sabbatical. We thank Victor Beresnevich, Manjul Bhargava, Tim Browning, John Cremona, James Maynard, Samir Siksek, Damiano Testa, Frank Thorne, and Stanley Xiao for helpful discussions. Finally, we are grateful to the anonymous referee for a careful reading and for particularly helpful comments.

\section{The cubic case}
\label{CubicCase}

In this section, we establish Theorem \ref{nonS3}. As discussed in the introduction, this is counting monic, $A_3$ cubic polynomials with integer coefficients bounded by $H$ in absolute value, and we will show that $N_{A_3,3} \ll H^{1.5+\eps}$. Let
\[
f(X) = X^3 + aX^2 + bX + c \in \bZ[X]
\]
be an irreducible cubic polynomial with $G_f \simeq A_3$ and $a, b, c \in [-H,H]$. Then its discriminant $\Del$ is a non-zero square. A short calculation reveals that
\[
\Del = a^2 b^2 - 4 b^3 - 4 a^3 c + 18 a b c - 27 c^2 = \frac{ 4I^3 - J^2}{27},
\]
where $I$ and $J$ are as defined in \eqref{IJdefCubic}. In particular, there exists $Y = 3 \sqrt{\Del} \in 3 \bN$ satisfying \eqref{dihedral1}.

\subsection{Parametrisation}

Let
\[
uv^3 = g = (J, Y),
\]
where $u,v \in \bN$ with $u$ cubefree, and let 
\[
\tilde g = uv^2.
\]
As $u$ is cubefree, observe that $\tilde g \mid 2I$. Write
\begin{equation} \label{CommonDivisors}
J = gx, \qquad Y = gy, \qquad 2 I = \tilde g z,
\end{equation}
where $x, y, z \in \bZ$ with $y > 0$ and $(x,y) = 1$. The equation \eqref{dihedral1} becomes
\begin{equation} \label{dihedral2}
2(x^2 + 3y^2) = uz^3.
\end{equation}

We factorise the left hand side of \eqref{dihedral2} in the ring $R := \bZ [\zeta]$ of Eisenstein integers, where $\zeta = \frac{-1+\sqrt{-3}}2$, giving
\begin{equation*}
2(x + y \sqrt{-3})(x - y \sqrt{-3}) = uz^3.
\end{equation*}
Note that $R$ is a principal ideal domain, and is therefore a unique factorisation domain. The greatest common divisor of $x + y \sqrt{-3}$ and $x - y \sqrt{-3}$ divides both $2x$ and $2y \sqrt{-3}$, and so it divides $2 \sqrt{-3}$. Write
\[
x + y \sqrt{-3} =  d \alp^3, \qquad x - y \sqrt{-3} = e \bet^3,
\]
for some $d, e, \alp, \bet \in R$ with $d,e$ cubefree. 

Note that $R$ has discriminant $-3$, so $3$ is the only rational prime that ramifies in $R$. Thus, either $u$ is cubefree in $R$, or else $u = 9u'$ for some cubefree $u' \in R$ not divisible by $\sqrt{-3}$. The \emph{cubefree component} of an element $\rho$ of $R$ is well defined up to multiplication by the cube of a unit, that is, up to sign: one prime factorises $\rho$ and divides by a maximal cubic divisor. Now $u$ is the cubefree component of $2de$, up to multiplication by $\pm 1$ or $\pm(\sqrt{-3})^3$. As $d,e \in R$ are cubefree and $\gcd(d,e) \mid 2 \sqrt{-3}$, we conclude that
\[
\frac{2de}u \in \{ A2^B \sqrt{-3}^C: A \in \{-1,1\}, B \in \{0,3\}, C \in \{-3,0,3\} \}.
\] 

Consider the norm
\[
N: \bQ(\sqrt{-3}) \to \bQ_{\ge 0}, \qquad q_1 + q_2 \sqrt{-3} \mapsto q_1^2 + 3 q_2^2,
\]
which in particular is multiplicative, and note that $R \subset \bQ(\sqrt{-3})$. As $N(d), N(e) \gg 1$ and $N(d)N(e) \ll N(u) = u^2$, we must have $N(d) \ll u$ or $N(e) \ll u$. Let us assume that $N(d) \ll u$; the other case $N(e) \ll u$ is similar.

As any element of $R$ is uniquely represented as a $\frac12 \bZ$-linear combination of $1$ and $\sqrt{-3}$, we may write
\[
d = \frac{q+r \sqrt{-3}}2, \qquad \alp = \frac{s+t \sqrt{-3}}2,
\]
with $q,r,s,t \in \bZ$, and so
\begin{equation} \label{param}
16x = q (s^3 - 9st^2) + 9r (t^3 - s^2t), \qquad 16y = 3q(s^2 t - t^3) + r (s^3 - 9 st^2).
\end{equation}
As $(x,y) = 1$, we must have $(s,t) \le 2$, and our bound $\frac{q^2 + 3r^2}4 = N(d) \ll u$ ensures that 
\[
q, r \ll \sqrt u.
\]

In fact we can say more. From \eqref{dihedral2} and \eqref{param}, we compute---using $N(\cdot)$ or otherwise--- that
\[
u(8z)^3 = 4(q^2+3r^2)(s^2+3t^2)^3.
\]
Recall that either $u$ is cubefree in $R$, or else $u = 9u'$ for some cubefree $u' \in R$ not divisible by $\sqrt{-3}$. Therefore $u$ is the cubefree component of $4(q^2+3r^2)$, up to multiplication by $\pm 1$ or $\pm(\sqrt{-3})^3$, and in particular $u \ll 4(q^2+3r^2)$. We already saw that $q^2 + 3r^2 \ll u$, so we conclude that
\begin{equation}\label{uzscales}
u \asymp q^2 + r^2, \qquad z \asymp s^2 + t^2.
\end{equation}

\subsection{Scales, and Lefton's approach}

We consider solutions for which $A \le |a| < 2A$, where $A \in [1,H]$ is a power of two. In the main part of the proof we only wish to choose the coefficient $a$ at the end, however it is convenient to fix the scale $A$ from the outset. There are $O(\log H)$ such scales. 

Lefton's approach \cite{Lef1979} is to choose $a \ll A$ and  $b \ll H$, and then to observe \cite[Lemma 2]{Lef1979} that the equation
\[
a^2 b^2 - 4 b^3 - 4 a^3 c + 18 a b c - 27 c^2 = 3Y^2
\]
has $O(H^\eps)$ integer solutions $(c,Y)$, uniformly in the relevant ranges. This shows that if $1 \le A \le H$ then there are $O(H^{1+\eps}A)$ solutions for which $a \ll A$. Thus, if $A \ll \sqrt H$ then there are $O(H^{1.5+\eps})$ solutions. \\

We assume henceforth that $999 \sqrt H < A \le H$ and $A \le |a| < 2A$. This ensures that $I = a^2 - 3b$ is positive, and that $I \asymp A^2$. Furthermore, we have
\[
\Del = a^2 b^2 - 4 b^3 - 4 a^3 c + 18 a b c - 27 c^2 \ll H^2 A^2.
\]
As $Y = 3 \sqrt \Del$, we may write this as
\[
Y \ll HA.
\]

We also choose scales $G, V, T \in \bN$, powers of 2, in $O( (\log H)^3)$ ways; these constrain our parameters to
\[
\tilde g \asymp G, \qquad |v| \asymp V, \qquad s^2 + t^2 \asymp T^2.
\]
Note from \eqref{CommonDivisors} and \eqref{uzscales} that
\begin{equation} \label{ScaleIdentity}
GT^2 \asymp I \asymp A^2.
\end{equation}

\bigskip

The plan is to count pairs $(I,J)$ of integers subject to the above ranges and satisfying \eqref{dihedral1} for some $Y \in \bN$ with $Y \ll HA$, and then to count $(a,b,c) \in \bZ^3$ with $|a| \asymp A$ and $|b|, |c| \le H$ corresponding to our choice of the pair $(I,J)$. We need a method that is efficient when $T$ is reasonably small, and another method that is efficient when $G$ is reasonably small. Note that
\[
q,r \ll \sqrt u \ll \sqrt G/V.
\] 

In the previous subsection, we saw that given $I,J$ with $(4I^3-J^2)/3$ a square there exist parameters $v,q,r,s,t$ with certain properties. The pair $(I,J)$ is determined in $O(H^{\eps})$ ways by $v,q,r,s,t$, uniformly in the relevant ranges. Indeed, the variables $x$ and $y$ are as in \eqref{param}, and $uz^3$ is then determined via \eqref{dihedral2}. Next, the variable $u$ is a divisor of $uz^3$, of which there are $O(H^{\eps})$, and finally we know $\tilde g, g, I, J$. The upshot is that we have reduced our task of counting pairs $(I,J)$ to that of upper bounding the number of quintuples $(v,q,r,s,t)$ that can possibly arise in this way.

\subsection{A linear instance of the concentration method}

From \eqref{CommonDivisors} and \eqref{param}, we have
\begin{equation} \label{Yconc}
GV |3q(s^2 t - t^3) + r (s^3 - 9 st^2)| \ll Y \ll HA.
\end{equation}

We begin by considering the case $s^2 t - t^3 = 0$. Since $(s,t) \le 2$, this case is only possible if $|s|,|t| \le 2$. There are $O(\sqrt G / V)$ possibilities for $q$ and $O(V)$ possibilities for $v$. See from the positivity of $y$ that $s^3 - 9st^2$ is a non-zero integer. Now \eqref{Yconc} implies that
\[
r \ll \frac{HA}{GV},
\]
so this case allows at most $O\Bigl(\Bigl( \frac{HA}{V \sqrt G} + \sqrt G \Bigr) H^\eps\Bigr)$ possibilities for the pair $(I,J)$.

We now assume that $s^2t - t^3 \ne 0$, whereupon
\[
q - \frac{r(s^3 - 9st^2)}{3(t^3 - s^2t)} \ll \frac{HA}{GV |t(t-s)(t+s)|}.
\]
The contribution from this case is therefore bounded above by 
\[
C_\eps H^\eps V \frac {\sqrt G}{V} \sum_{\substack{s,t \ll T \\ t \notin \{-s, 0, s\}}} \Bigl( \frac{HA}{GV |t(t-s)(t+s)|} + 1 \Bigr)
\ll H^{2\eps} \Bigl( \frac{HA}{V\sqrt G} + T^2 \sqrt G \Bigr).
\]
By \eqref{ScaleIdentity} we conclude that there are $O(H^{1+\eps}T)$ possibilities for $(I,J)$ in total.

\subsection{Root separation} The approach in the previous subsection is effective when $T$ is reasonably small. Here we develop an approach that works well when $G$ is reasonably small. We assume that $|t| \ge |s|$, so that $|t| \asymp T$; the other scenario is similar. We begin by choosing $v \ll V$ and $q \ll \sqrt G/ V$.

We begin with the case $r \ne 0$. Choose $r \ne 0$ with $r \ll \sqrt G /V$, define a polynomial $\cF$ by
\[
\cF(X) = rX^3 + 3q X^2 - 9r X - 3q,
\]
and write $\kap = s/t$. From \eqref{Yconc} we obtain
\begin{equation} \label{AuxCub}
\cF(\kap) = r \kap^3 + 3q \kap^2 - 9r \kap - 3q \ll \frac{HA}{GVT^3}.
\end{equation}
Using what is now known as the \emph{Mahler measure} \cite{Smy2008}, Mahler analysed the separation of roots of polynomials. It is this that enables us to capitalise efficiently on the concentration inherent in the cubic inequality \eqref{AuxCub}. Mahler established, in particular, a lower bound for the minimum distance between two roots, in terms of the degree, discriminant, and the sum of the absolute values of the coefficients of the polynomial \cite[Corollary 2]{Mah1964}. Applying this to the polynomial $\cF$ with roots $\kap_1, \kap_2, \kap_3$ yields
\[
\min_{1 \le i < j \le 3} |\kap_i - \kap_j| \gg (\disc \: \cF)^{1/2} (|q|+|r|)^{-2}.
\]

One might not immediately realise that the discriminant of $\cF$ should necessarily be positive and fairly large. However, this is indeed the case, and it happens to be a constant multiple of $N(d)^2$. From the formula for the discriminant of a cubic polynomial, we compute that
\begin{align*}
\disc \: \cF &= (3q)^2(-9r)^2 - 4r(-9r)^3 - 4(3q)^3(-3q) - 27r^2 (-3q)^2 + 18r(3q)(-9r)(-3q) \\
&=  (18(q^2+3r^2))^2 \gg (|q|+|r|)^4.
\end{align*}
We now have
\[
\min_{1 \le i < j \le 3} |\kap_i - \kap_j| \gg 1.
\]

As
\[
\prod_{i \le 3}| \kap - \kap_i| \ll \frac{HA}{rGVT^3},
\]
there must therefore exist $i \in \{1,2,3 \}$ such that
\[
\kap - \kap_i \ll \frac{HA}{rGVT^3},
\]
and so
\[
s - \kap_i t \ll \frac{HA}{rGVT^2}.
\]
The upshot is that the other parameters determine $O\bigl(\frac{HA}{rGVT^2} + 1 \bigr)$ possibilities for $s$. Bearing in mind \eqref{ScaleIdentity}, this case contributes at most
\[
C_\eps H^\eps V \frac{\sqrt G}{V} \sum_{0 < |r| \ll \sqrt G/V} T \Bigl(\frac{HA}{rGVT^2} + 1 \Bigr) \ll H^\eps(H^{1+\eps} + A \sqrt G)
\]
solutions.

If instead $r=0$, then \eqref{uzscales} implies that $q \asymp \sqrt G /V$, and with $\kap = s/t$ we obtain
\[
\kap^2 - 1 \ll \frac{HA}{(GT^2)^{3/2}} \ll \frac H{A^2}.
\]
Then
\[
|\kap| - 1 \ll \frac H{A^2},
\]
and so 
\[
|s| - |t| \ll \frac{HT}{A^2}.
\]
This case permits at most
\[
C_\eps H^\eps V \frac{\sqrt G}V T \Bigl( \frac{HT}{A^2}+1 \Bigr) \ll H^{1+\eps}
\]
solutions.

We conclude that there are $O(H^{\eps} (H+ A \sqrt G))$ possibilities for the pair $(I,J)$.

\subsection{An approximately quadratic inequality}

From the previous two subsections, we glean that the number of allowed pairs $(I,J)$ is at most
\begin{align*}
C_\eps H^\eps \min \{ HT, H + A \sqrt G\} &\ll H^\eps (H+ (HT)^{1/2} (A \sqrt G)^{1/2}) \\ &\ll H^\eps (H + A \sqrt H) \ll H^\eps A \sqrt H,
\end{align*}
since $A \gg \sqrt H$.\\

Now suppose that we have chosen $I$ and $J$, with $0 < I \asymp A^2$ and $4I^3 - J^2 > 0$. Our final task is to count the number of triples $(a,b,c)$ of integers such that $a \asymp A$ and $b,c \ll H$, and satisfying the equations \eqref{IJdefCubic}. The idea is to extract concentration from the inequalities $b \ll H$ and $c \ll H$.

We have
\[
a^2 - I = 3b \ll H,
\]
so the shifted integer variable $x = |a| - \sqrt I$ will necessarily be small, and in the first instance
\[
x \ll \frac{H}{|a| + \sqrt{I}} \ll \frac H{||a| - \sqrt I|} = \frac H{|x|},
\]
so $x \ll \sqrt H$. There are at most two solutions $(a,b,c)$ with $x = 0$, so we assume in the sequel that $x \ne 0$. Now
\[
|x(|a| + \sqrt I)| = |3b| \ge 3,
\]
so $x \gg H^{-1}$. We introduce a scale $X \in \bR_{>0}$, of the form $2^m$ for some integer $m$, with $H^{-1} \ll X \ll \sqrt H$, and consider solutions $(a,b,c)$ with $|x| \asymp X$. There are $O(\log H)$ possibilities for the scale $X$, and
\[
X \ll \frac{H}{|a| + \sqrt I} \ll \frac{H}{A}.
\]

We also have
\[
J - 3aI +a^3 = 27c \ll H.
\]
As $A > 999 \sqrt H$, we know that $J = 27c - 9ab + 2a^3$ and $a$ have the same sign, so
\[
|J| - 2I^{3/2} + 3 \sqrt I x^2 + x^3 = |J| - 3|a| I + |a|^3 \ll H.
\]
The left hand side above is cubic in $x$, but $x$ is fairly small, so we can approximate the cubic by a quadratic in order to exploit concentration. The triangle inequality gives
\[
x^2 - x_0^2 \ll \frac{X^3 + H}A,
\]
where 
\[
x_0 = \sqrt{ \frac{2I^{3/2} - |J|}{3 \sqrt I}}.
\]
Observe that $x_0$ is a positive real number, since 
\[
(2I^{3/2} + |J|) (2I^{3/2} - |J|) =4I^3 - J^2 > 0.
\]
Now
\[
|x| - x_0 \ll \frac{X^3 + H}{A(|x|+x_0)} \ll \frac{X^3 + H}{AX} = \frac{X^2}A + \frac{H}{AX}.
\]

Recall that $x \in \bZ - \sqrt I$ is a discrete variable. The number of possibilities for $x$ is therefore bounded above by a constant times
\[
\min \Bigl \{X, \frac{X^2}A + \frac{H}{AX} + 1\Bigr \} \ll \frac{X^2}A + \sqrt X \sqrt{\frac H {AX}} + 1 \ll \frac{H^2}{A^3} + \sqrt{\frac H A}.
\]
Once we know $x$, the triple $(a,b,c)$ is determined in at most two ways. The total number of monic, $A_3$ cubics with $|a| \asymp A$ is therefore bounded above by
\[
C_\eps H^\eps  A \sqrt H \Biggl( \frac{H^2}{A^3} + \sqrt { \frac H A} \Biggr) \ll \frac{H^{2.5+\eps}}{A^2} + H^{1+\eps} \sqrt A \ll H^{1.5+\eps},
\]
since $\sqrt H \ll A \ll H$, and this completes the proof of Theorem \ref{nonS3}.

\section{A remarkable symmetry}
\label{res}

In this section, we establish \eqref{GroupUpper}. Theorem \ref{KW} tells us that if $f$ is irreducible and $G_f$ is isomorphic to $D_4, V_4$ or $C_4$ if and only if the cubic resolvent
\[
r(X) = r(X;a,b,c,d) = X^3 - bX^2 + (ac-4d)X - (a^2d - 4bd + c^2)
\]
has an integer root. Moreover, it follows from the triangle inequality that if $H \ge 150$, $f \in \cS_H$ and $r(x) = 0$ then $|x| \le 2H$. The proposition below therefore implies \eqref{GroupUpper}.

\begin{prop} \label{counting}
Write $R(H)$ for the number of integer solutions
\[
(x,a,b,c,d) \in [-2H,2H] \times [-H,H]^4
\]
to the equation
\begin{equation} \label{orig}
 r(x;a,b,c,d) = 0.
\end{equation}
Then 
\[
R(H) \ll H^2 (\log H)^2.
\]
\end{prop}

We set about proving this. Multiplying \eqref{orig} by 4, we obtain
\begin{equation} \label{MainEq}
(x^2 - 4d) \cdot (a^2 - 4(b-x)) = (xa - 2c)^2.
\end{equation}
Change variables, replacing $b-x$ by $e$, so that \eqref{MainEq} becomes
\begin{equation} \label{symmetry}
(x^2 - 4d) \cdot (a^2 - 4e) = (xa-2c)^2,
\end{equation}
with $|e| \le 3H$. Observe that the equation \eqref{symmetry} exhibits a great deal of symmetry. We need to count integer solutions $(x,a,c,d,e)$ with
\[
|a|,|c|,|d| \le H, \qquad |x| \le 2H, \qquad |e| \le 3H.
\]

We begin with the case in which both sides of \eqref{symmetry} are 0. For each $c$ there are at most $\tau(2c)$ choices of $(x,a)$. Therefore, by an average divisor function estimate, the number of choices of $(x,a,c)$ is $O(H \log H)$. Having chosen $x,a,c$ with $xa = 2c$, there are then $O(H)$ possible $(d,e)$. We conclude that the number of solutions for which $xa = 2c$ is $O(H^2 \log H)$. It remains to treat solutions for which $xa \ne 2c$.\\

Write $x^2-4d = uv^2$ with $u \in \bZ \setminus \{0\}$ squarefree and $v \in \bN$. This forces $a^2-4e = uw^2$ and $xa-2c =\pm uvw$ for some $w \in \bN$. Our strategy will be to upper bound the number of lattice points $(u,v,w,x,a)$ with $u \ne 0$ in the region defined by $|x|, |a| \le 2H$ and
\begin{align}
\label{ineq1} |x^2 - uv^2| &\le 12H \\
\label{ineq2} |a^2 - uw^2| &\le 12H \\
\label{ineq3} \min \{ |xa - uvw|, |xa + uvw| \} &\le 2H.
\end{align}
At most two values of $(c,d,e)$ are then determined by $(u,v,w,x,a)$.

For the case $u < 0$, choose $p=-u$ in the range $1 \le p \ll H$. Then \eqref{ineq1} implies $x^2 + pv^2 \ll H$, which has $O(H/\sqrt p)$ solutions $(x,v)$. Similarly there are $O(H/\sqrt p)$
choices of $(a,w)$. As
\[
\sum_{1 \le p \ll H} H^2/p \ll H^2 \log H,
\]
we find that the total contribution from this case is $O(H^2 \log H)$.

It remains to deal with the case $u > 0$. Arguing by symmetry, it suffices to count solutions for which 
\begin{equation*}
u> 0, \qquad x,a \ge 0, \qquad 1 \le w \le v.
\end{equation*}
Now \eqref{ineq3} is equivalent to
\begin{equation} \label{ineq3nice}
|xa - uvw| \le 2H.
\end{equation}

Choose $u$ and $v$ to begin with, so that $uv^2 \ll H^2$. First suppose $uv^2 \le 40H$. Then $x,a \ll \sqrt H$, so the contribution from this case is bounded above by a constant times
\[
H \sum_{v \le \sqrt {40H}} \sum_{u \le 40H/v^2} \sum_{w \le v} 1 \ll H^2 \log H.
\]
This is more than adequate, so in the sequel we assume that $uv^2 > 40H$.

Now \eqref{ineq1} implies that $x \asymp v \sqrt u$. There are $v$ choices of $w$, and since
\[
|x - v \sqrt u|  \le \frac{12H}{x+v \sqrt u} \ll \frac H{v \sqrt u}
\]
there are $O ( 1 +\frac H{v\sqrt u} ) = O (\frac H{v\sqrt u} )$ choices of $x$. Using \eqref{ineq3nice}, observe that
\[
x(a-w\sqrt u) +w\sqrt u (x - v \sqrt u) = xa - uvw \ll H.
\]
As $w \le v$, we now have
\[
a - w\sqrt u \ll  \frac{H}{v \sqrt u} + w\sqrt u \frac{|x^2 - uv^2|}{(x+v \sqrt u)^2} \ll \frac{H}{v \sqrt u}.
\]
In particular, there are $O ( 1 +\frac H{v\sqrt u} ) = O (\frac H{v\sqrt u} )$ possibilities for $a$. We obtain the upper bound
\[
\sum_{1 \le u,v \ll H^2} v \Bigl (\frac H{v\sqrt u} \Bigr)^2
= H^2 \sum_{1 \le u,v \ll H^2} \frac1{uv} \ll H^2 (\log H)^2,
\]
completing the proof.

\section{A construction}
\label{cons}

In this section, we establish \eqref{LowerBound}. Our construction is motivated by the previous section. Let $\del$ be a small positive constant. We shall choose positive integers
\[
x,a,u,w \equiv 12 \mmod 18, \qquad v \equiv 4 \mmod 6
\]
with $u$ squarefree, in the ranges
\begin{align*}
1 &\le u \le  H^{2-2\del} \\
\del^{-1} \sqrt H &\le \frac12 v \sqrt u \le w \sqrt u \le v \sqrt u \le \del^2 H\\
v\sqrt u &< x \le v \sqrt u + \frac {\del H}{v \sqrt u}\\
w\sqrt u &< a \le w \sqrt u + \frac {\del H}{v \sqrt u}.
\end{align*}

Let us now bound from below the number of choices $(u,v,w,x,a)$. If we choose $u,v \in \bN$ with $u \le H^{2-2\del}$, $v \ge 99$ and
\[
2 \del^{-1} \sqrt H \le v \sqrt u \le \del^2 H,
\]
then the number of choices for $(w,x,a)$ is bounded below by a constant times $v (\frac H{v \sqrt u} )^2 = \frac{H^2}{uv}$. Thus, the number of possible choices of $(x,a,u,v,w)$ is bounded below by a constant times
\[
X(H) := H^2 \sum_{u \in \cU} u^{-1} \sum_{v \in \cV(u)} v^{-1},
\]
where
\[
\cU = \{u \in \bN: |\mu(u)| = 1, \: u \equiv 12 \mmod 18, \: u \le H^{2-2\del} \}
\]
and
\[
\cV(u) = \{ v \ge 99: v \equiv 4 \mmod 6, \: 2 \del^{-1} \sqrt H \le v \sqrt u \le \del^2 H\}.
\]
We compute that
\[
X(H) = H^2 \sum_{u \in \cU} u^{-1} \sum_{v \in \cV(u)} v^{-1}  \gg H^2 \log H \sum_{u \in \cU} u^{-1}.
\]

Observe that the conditions 
\[
u \equiv 12 \mmod 18, \qquad |\mu(u)| = 1
\]
on $u$ are equivalent to the conditions
\[
r \equiv 5 \mmod 6, \qquad |\mu(r)| = 1
\]
on $r = u/6$. It thus follows from work of Hooley \cite[Theorem 3]{Hoo1975} that
\[
\# \{u \in \cU: u \le t\} = c_0 t + O(\sqrt t),
\]
for some constant $c_0 > 0$. Partial summation now gives
\[
\sum_{u \in \cU} u^{-1} \sim c_1 \log H,
\]
where $c_1 = c_1(\del) = (2-2\del)c_0$, so in particular $X(H) \gg H^2 (\log H)^2$.

Given such a choice of $(u,v,w,x,a)$, define $b,c,d \in \bZ$ by
\[
4d = x^2 - uv^2, \qquad 4(b-x) = a^2 - uw^2, \qquad 2c = xa - uvw.
\]
We claim that the polynomial $f$ defined by \eqref{consider} lies in $\cS_H$, and that $G_f$ is isomorphic to $D_4$, $V_4$ or $C_4$. We now confirm this claim.

Plainly $|a| \le H$. Moreover, since
\[
4d = x^2 - uv^2 = (x- v \sqrt u)(x+v \sqrt u),
\]
we have
\[
0 < 4d \le \frac{\del H}{v \sqrt u} \Bigl(2v \sqrt u + \frac {\del H}{v \sqrt u} \Bigr) < H,
\]
and similarly $0 < 4(b-x) < H$. Now the triangle inequality gives $|b| \le x + H/4 < H$. Finally, we check that
\[
0 < 2c = xa - uvw \le \Bigl(v \sqrt u + \frac {\del H}{v \sqrt u} \Bigr) 
\Bigl(w \sqrt u + \frac {\del H}{v \sqrt u} \Bigr) - uvw < H.
\]
We have shown that $|a|, |b|, |c|, |d| \le H$.

Since $x,a$ and $u$ are divisible by $3$, we have $a \equiv b \equiv c \equiv d \equiv 0 \mmod 3$. Furthermore
\[
4d = x^2 - uv^2 \equiv -3v^2 \mmod 9,
\]
so $9 \nmid d$. Thus, by Eisenstein's criterion, the polynomial \eqref{consider} is irreducible. Hence $f \in \cS_H$. Moreover, since $x \in \bZ$ is a root of the cubic resolvent of $f$, we know from Theorem \ref{KW} that $G_f$ is isomorphic to $D_4$, $V_4$ or $C_4$.

Finally, we verify that the number of distinct polynomials $f(X)$ arising from this construction is at least a constant times $H^2 ( \log H)^2$. We achieve this by showing that a polynomial $f(X)$ occurs for at most three different choices of $(u,v,w,x,a)$. Suppose the quadruple $(a,b,c,d)$ is obtained via this construction. Then $x$ is a root of the cubic resolvent of $f$, so there are at most three possibilities for $x$. Since $u,v,w \in \bN$ with $u$ squarefree, the equations
\[
x^2-4d = uv^2, \qquad a^2 - 4(b-x) = uw^2
\]
now determine the triple $(u,v,w)$. Thus, a quadruple $(a,b,c,d)$ can be obtained from $(u,v,w,x,a)$ in at most three ways via our construction, and so we've constructed at least a constant times $H^2 (\log H)^2$ polynomials in this way. This completes the proof of \eqref{LowerBound}.

\section{$V_4$ and $C_4$ quartics}
\label{V4C4}

In this section we prove Theorem \ref{thm2}, and thereby also establish Theorem \ref{thm1}. From \S \ref{res}, we know that if $f \in \cS_H$ and $G_f$ is isomorphic to $V_4$ or $C_4$ then, with $O(H^2 \log H)$ exceptions, there exist integers $u,v,w > 0$ and $x \in [-2H, 2H]$ such that
\begin{equation} \label{sub} 
d = \frac{x^2 - uv^2}4, \qquad b = x + \frac{a^2 - uw^2}4, \qquad c = \frac{xa \pm uvw}2.
\end{equation}

\subsection{$V_4$ quartics} \label{SV4}

By Theorem \ref{KW}, the discriminant $\Del$ of $f$ is a square. We have the standard formula \cite[\S 14.6]{DF2004}
\begin{align*}
\Del &= -128b^2d^2 - 4a^3c^3 + 16b^4d - 4b^3c^2 - 27a^4d^2 + 18abc^3 \\
&\quad + 144a^2bd^2 - 192acd^2 + a^2b^2c^2 - 4a^2b^3d - 6a^2c^2d \\
&\quad + 144bc^2d + 256d^3 - 27c^4 - 80ab^2cd + 18a^3bcd.
\end{align*}
We make the substitutions \eqref{sub} using the software \emph{Mathematica} \cite{Mathematica}, obtaining the factorisation
\begin{align} \notag
\frac{64 \Del} {u^2 (2v^2  \pm a v w + w^2 x)^2}
&= a^4 - 64 u v^2 \mp 32 a u v w - 2 a^2 u w^2  \\
\label{star} &\qquad + u^2 w^4 - 16 a^2 x - 16 u w^2 x + 64 x^2.
\end{align}
Note that the denominator of the left hand side is non-zero, for the irreducibility of $f$ implies that $\Del \ne 0$. We now equate the right hand side with $y^2$, for some $y \in \bZ$. Given $u,v,w,a$, the integer point $(x,y)$ must lie on one of the two curves $C^\pm_{u,v,w,a}$ defined by
\begin{equation} \label{sieve1}
(8x-(a^2+uw^2))^2 - (4a^2uw^2+ 64uv^2 \pm 32 auvw) = y^2.
\end{equation}
Therefore $N_{V_4}$ is bounded above, up to a multiplicative constant, by $H^2 \log H$ plus the number of sextuples $(u,v,w,x,a,y) \in \bN^3 \times \bZ^3$ satisfying $|x|,|a| \le 8H$, \eqref{ineq1}, \eqref{ineq2}, \eqref{ineq3} and $(x,y) \in C^+_{u,v,w,a} \cup C^-_{u,v,w,a}$.\\

We first consider the contribution from $(u,v,w,a)$ for which $C_{u,v,w,a}^\pm$ is reducible over $\overline{\bQ}$. In this case
\[
(8x-(a^2+uw^2))^2 - (4a^2uw^2+ 64uv^2 \pm 32 auvw)
\]
is a square in $\overline{\bQ}[x]$, so
\[
4a^2uw^2+ 64uv^2 \pm 32 auvw = 0.
\]
As $u \ne 0$ we now have $(aw \pm 4v)^2 = 0$, so
\begin{equation} \label{reducible}
aw = \mp 4v.
\end{equation}

\begin{enumerate}
\item For the case $uw^2 \le 40H$, we first choose $u \in [1, 40H]$, then there are $O(\sqrt{H / u})$ choices of $w$, and by \eqref{ineq2} there are $O(\sqrt H)$ possibilities for $a$. This then determines at most two possible $v$, via \eqref{reducible}. Since 
\[
|x| - v \sqrt u \ll \frac H {|x| + v \sqrt u},
\]
there are now $O(1 + H/\sqrt u) = O(H/\sqrt u)$ choices of $x$. The contribution from this case is therefore bounded above by a constant times
\[
\sum_{u \le 40H} \sqrt{\frac H u } \sqrt H  \frac H {\sqrt u} \ll H^2 \log H.
\]
\item If instead $uw^2 > 40H$, then $|a| \asymp w \sqrt u$, so from \eqref{reducible} we have
\[
v \gg |aw| \gg w^2 \sqrt u.
\]
Start by choosing $u,w$ for which $40H < uw^2 \ll H^2$. There are then 
\[
O \Bigl(1 + \frac H {w \sqrt u} \Bigr) = O \Bigl( \frac H {w \sqrt u} \Bigr)
\]
possible $a$, since
\[
|a| - w \sqrt u \ll \frac H{|a| + w \sqrt u},
\]
and then $v$ is determined by \eqref{reducible} in at most two ways. Now
\[
|x| - v \sqrt u \ll \frac H{v \sqrt u},
\]
so the number of possibilities for $x$ is bounded above by a constant times
\[
1 + \frac{H}{v \sqrt u} \ll \frac{H}{v \sqrt u} \ll \frac H{w^2 u}.
\]
Thus, the contribution from this case is bounded above by a constant times
\[
\sum_{uw^2 \ll H^2} \frac H {w \sqrt u} \cdot \frac H{w^2 u} \ll H^2.
\]
\end{enumerate}
We have shown that there are $O(H^2 \log H)$ sextuples 
\[
(u,v,w,x,a,y) \in \bN^3 \times \bZ^3
\]
satisfying $|x|,|a| \le 8H$, \eqref{ineq1}, \eqref{ineq2}, \eqref{ineq3} and \eqref{sieve1} such that $C_{u,v,w,a}^\pm$ is reducible over $\overline{\bQ}$.\\

It remains to address the situation in which $C_{u,v,w,a}^\pm$ is absolutely irreducible. We will ultimately apply Vaughan's uniform count for integer points on curves of this shape \cite[Theorem 1.1]{Vau2014}.

Suppose $w \le v$ and $uv^2 \le 40 H$. Then $x,a \ll \sqrt H$, so the number of solutions is bounded above by a constant times
\[
H \sum_{v \le \sqrt{40H}} \sum_{u \le 40H/v^2} \sum_{w \le v}1 \ll H^2 \log H.
\]
Similarly, if $v \le w$ and $uw^2 \le 40 H$ then there are $O(H^2 \log H)$ solutions.

Next, we consider the scenario in which $w \le v$ and $uv^2 > 40 H$. Using \eqref{ineq1}, this implies
\[
x^2 > \frac12uv^2,
\]
so $|x| > \frac12 v\sqrt u$. Using \eqref{ineq3} gives
\[
||x|(|a|-w\sqrt u) +w\sqrt u(|x|-v \sqrt u)| = ||xa| - uvw| \le 2H.
\]
As $|x| > \frac12 v\sqrt u$ and $w \le v$, we now have
\[
||a| - w \sqrt u| \le \frac {2H + w \sqrt u ||x|-v \sqrt u|}{|x|} \le \frac{4H}{v \sqrt u} + 2 ||x|-v\sqrt u|.
\]
Since
\begin{equation} \label{xstrong}
||x|-v\sqrt u|= \frac{|x^2 - uv^2|}{||x| + v \sqrt u|} \le \frac{12H}{v \sqrt u},
\end{equation}
we arrive at the inequality
\begin{equation*}
|a| - w \sqrt u \ll \frac{H}{v \sqrt u}.
\end{equation*}
In particular, given $u, v, w$ there are 
\[
O \Bigl(1+\frac{H}{v \sqrt u} \Bigr) = O \Bigl(\frac{H}{v \sqrt u} \Bigr)
\]
possibilities for $a$.

Choose $u,v,w \in \bN$ and $a \in \bZ$ such that $C_{u,v,w,a}^\pm$ is absolutely irreducible. Note \eqref{xstrong}, and put $L = \frac{12H}{v \sqrt u} + 1$. Now \cite[Theorem 1.1]{Vau2014} reveals that \eqref{sieve1} has $O(L^{1/2})$ solutions $(x,y)$, with an absolute implied constant. As $w \le v$, the number of solutions is therefore bounded by a constant multiple of
\begin{align*}
\sum_{uv^2 \ll H^2} v \frac{H}{v \sqrt u} \sqrt{\frac{H}{v \sqrt u}} &\ll H^{3/2} \sum_{u \ll H^2} u^{-3/4} \sum_{v \ll H/\sqrt u} v^{-1/2} \\
&\ll H^2 \sum_{u \ll H^2} u^{-1} \ll H^2 \log H.
\end{align*}

The final case, wherein $v \le w$ and $uw^2 > 40H$, is very similar to the previous one. We have considered all cases, and conclude that
\[
N_{V_4} \ll H^2 \log H.
\]

\subsection{$C_4$ quartics}

We follow a similar strategy to the one that we used for $V_4$. The root of the cubic resolvent is $x$, so from Theorem \ref{KW} we find that $(x^2 - 4d)\Del$ is a perfect square. Observe from \eqref{sub} that $x^2 - 4d = uv^2$. Factorising the right hand side of \eqref{star}, we thus obtain
\[
u \Bigl((8x-(a^2+uw^2))^2 -  4u  (aw \pm 4v)^2 \Bigr) = y^2,
\]
for some $y \in \bZ$. Given $u,v,w,a$, this defines a pair of curves $Z_{u,v,w,a}^\pm$. As $u \ne 0$, the curve $Z_{u,v,w,a}^\pm$ is absolutely irreducible if and only if the curve $C^\pm_{u,v,w,a}$ defined in \eqref{sieve1} is absolutely irreducible. The remainder of the proof can be taken almost verbatim from \S \ref{SV4}. We conclude that
\[
N_{C_4} \ll H^2 \log H,
\]
and this completes the proof of Theorem \ref{thm2}. In light of \eqref{GroupUpper} and \eqref{LowerBound}, we have also completed the proof of Theorem \ref{thm1}.

\section{$A_4$ quartics}
\label{A4}

In this section, we establish Theorem \ref{thm3}. We again use Theorem \ref{KW}, which in particular asserts that $A_4$ quartics have square discriminant. It remains to show that the diophantine equation
\[
\disc(X^4+aX^3+bX^2+cX+d) = y^2
\]
has $O(H^{\frac52+\frac1{\sqrt6}+\eps})$ integer solutions for which $|a|,|b|,|c|,|d| \le H$ and $y \in \bZ \setminus \{ 0 \}$. We have the standard formula \cite{BS2015}
\[
\Del := \disc(X^4+aX^3+bX^2+cX+d) = \frac{4I^3-J^2}{27},
\]
where $I$ and $J$ are as defined in \eqref{IJdef}. The idea now is to count integer triples $(I,J,y)$ solving \eqref{IJeq} with $I \ll H^2$ and $y \ne 0$, and to then count quadruples of integers $(a,b,c,d) \in [-H,H]^4$ corresponding via \eqref{IJdef} to a given $(I,J)$. Each integer $I \ll H^2$ defines via \eqref{IJeq} a quadratic polynomial in $(J,y)$ with non-zero discriminant. Thus, by \cite[Lemma 2]{Lef1979}, the diophantine equation \eqref{IJeq} admits $O(H^{2+\eps})$ solutions $(I,J,y)$ with $I \ll H^2$. It therefore remains to show that if $4I^3 - J^2 \ne 0$ then there are $O(H^{\frac12+\frac1{\sqrt6}+\eps})$ integer quadruples $(a,b,c,d) \in [-H,H]^4$ satisfying \eqref{IJdef}.

Fix $I,J$ for which $4I^3 \ne J^2$. From \eqref{IJdef}, we have
\begin{align*}
J + 2bI =  96bd + 3abc - 27c^2 - 27a^2d.
\end{align*}
Therefore
\[
J + 27c^2 + 27a^2d = b(96d + 3ac - 2I),
\]
and so
\[
(J + 27c^2 + 27a^2d )^2 = b^2 (96d + 3ac - 2I)^2  = (I - 12d + 3ac)(96d + 3ac - 2I)^2.
\]
Writing
\[
g(a,c,d) = (I - 12d + 3ac)(96d + 3ac - 2I)^2 - (J + 27c^2 + 27a^2d )^2,
\]
the equation $g(a,c,d) = 0$ cuts out an affine surface $Y_{I,J}$. It remains to show that there are $O(H^{\frac12+ \frac1{\sqrt6}+\eps})$ integer solutions $(a,c,d) \in [-H,H]^3$ to $g(a,c,d) = 0$.

\begin{lemma} \label{NoLines}
The affine surface $Y_{I,J}$ contains no rational lines.
\end{lemma}

\begin{proof} 
A line has the form
\[
\cL = \{ (\alp, \gam, \del) + t (A,C,D) : t \in \bQ \}
\]
for some $(\alp, \gam, \del) \in \bQ^3$ and some $(A,C,D) \in \bQ^3 \setminus \{\bzero\}$. There are three types of line to consider:
\begin{enumerate}[I.]
\item $\cL = \{ (0, \gam, \del) + t (1,C,D) : t \in \bQ \}$;
\item $\cL = \{ (\alp, 0 , \del) + t (0,1,D) : t \in \bQ \}$;
\item $\cL = \{ (\alp, \gam, 0) + t (0,0,1) : t \in \bQ \}$.
\end{enumerate}

\bigskip

In each case, we substituted the form of the line into $g(a,c,d) = 0$ and expanded it as a polynomial in $t$. Equating coefficients then provided seven equations.

In Case I, we used the software \emph{Mathematica} \cite{Mathematica} to obtain $4I^3 - J^2 = 0$ by elimination of variables. The proof reveals, in fact, that there are no complex lines, but all we need is for there to be no rational lines. Here is the code.

\begin{Verbatim} [fontsize=\scriptsize]

a = t; c = \[Gamma] + t q; d = \[Delta] + t r;
Collect[Expand[(k - 12 d + 3 a c) (96 d + 3 a c - 2  k)^2 - (j + 
      27 c^2 + 27 a^2 d)^2 ] , t]
Eliminate[27 q^3 - 729 r^2 == 0 && 162 q^2 r + 81 q^2 \[Gamma] - 1458 r \[Delta] ==  0 && 
-27 k q^2 - 729 q^4 + 20736 q r^2 + 324 q r \[Gamma] +  81 q \[Gamma]^2 + 162 q^2 \[Delta] 
- 729 \[Delta]^2 == 0 && -54 j r - 432 k q r - 110592 r^3 - 54 k q \[Gamma] - 2916 q^3 \[Gamma] 
+ 20736 r^2 \[Gamma] + 162 r \[Gamma]^2 + 27 \[Gamma]^3 + 41472 q r \[Delta] 
+ 324 q \[Gamma] \[Delta] ==  0 && -54 j q^2 + 13824 k r^2 - 432 k r \[Gamma] - 27 k \[Gamma]^2 
- 4374 q^2 \[Gamma]^2 - 54 j \[Delta] - 432 k q \[Delta] - 331776 r^2 \[Delta] 
+ 41472 r \[Gamma] \[Delta] + 162 \[Gamma]^2 \[Delta] + 20736 q \[Delta]^2 ==  0 && -432 k^2 r
- 108 j q \[Gamma] - 2916 q \[Gamma]^3 + 27648 k r \[Delta] - 432 k \[Gamma] \[Delta] 
- 331776 r \[Delta]^2 + 20736 \[Gamma] \[Delta]^2 == 0 && -j^2 + 4 k^3 - 54 j \[Gamma]^2 - 729
\[Gamma]^4 - 432 k^2 \[Delta] + 13824 k \[Delta]^2 - 110592 \[Delta]^3 == 0, 
{\[Gamma], \[Delta], q, r}]

\end{Verbatim}

In Case II the $t^4$ coefficient is $-729$, and in Case III the $t^3$ coefficient is $-110592$, so these cases can never occur. We have deduced $4I^3 - J^2 = 0$ from the existence of a rational line, completing the proof.
\end{proof}

Observe that $Y_{I,J}$ is the zero locus of the polynomial
\[
g(a,c,d) = c_3 d^3 + c_2(a,c) d^2 + c_1(a,c)d + c_0(a,c),
\]
where
\begin{align*}
c_3 &= -110592, \qquad c_2(a,c) = -729 a^4 + 20736 a c + 13824 I, \\
c_1 (a,c) &=  162 a^2 c^2 - 54 a^2 J - 432 a c I - 432 I^2, \\
 c_0(a,c) &= 27 a^3 c^3 - 729 c^4 - 54 c^2 J -  J^2 - 27 a^2 c^2 I + 
 4 I^3.
\end{align*}

\begin{lemma} \label{AbsIrred}
The affine surface $Y_{I,J}$ is absolutely irreducible.
\end{lemma}

\begin{proof} Assume for a contradiction that $Y_{I,J}$ is not absolutely irreducible. Then there exist polynomials $f_0(a,c)$, $g_0(a,c)$, and $h_0(a,c)$, defined over $\overline \bQ$, for which
\[
c_3 d^3 + c_2(a,c) d^2 + c_1(a,c)d + c_0(a,c) = (c_3 d^2 + f_0(a,c)d + g_0(a,c))(d + h_0(a,c)).
\]
Now
\begin{equation} \label{AbsIrr0}
c_2(a,c) =  f_0(a,c) + c_3 h_0(a,c) ,
\end{equation}
\begin{equation} \label{AbsIrr1}
c_1(a,c) = g_0(a,c) + f_0(a,c) h_0(a,c),
\end{equation}
and
\begin{equation}\label{AbsIrr2}
c_0(a,c) = g_0(a,c) h_0(a,c).
\end{equation} 
From \eqref{AbsIrr0} we have
\[
\max\{\deg_a(f_0), \deg_a(h_0)\} \ge \deg_a(c_2)  = 4.
\]
From \eqref{AbsIrr2}, we have $g_0, h_ 0 \ne 0$ and
\[
\deg_a(g_0) \le \deg_a(c_0) = 3.
\]
Unless $f_0 = 0$, these two inequalities together violate \eqref{AbsIrr1}, since $\deg_a(c_1) = 2$. Finally, if $f_0 = 0$ then $\deg_a(h_0) = 4$, violating \eqref{AbsIrr2}. This contradiction confirms that $Y_{I,J}$ is absolutely irreducible.
\end{proof}

Finally, we complete the proof of Theorem \ref{thm3}. By \cite[Lemma 1]{Bro2011}, there exist polynomials $g_1, \ldots, g_\cJ \in \bZ[a,b,d]$ with $\cJ \ll H^{\frac1{\sqrt6}+\eps}$, and a finite set of points $Z \subseteq Y_{I,J}$ such that
\begin{enumerate}
\item Each $g_j$ is coprime to $g$, and has degree $O(1)$;
\item $|Z| \ll H^{\frac2{\sqrt 6} + \eps}$;
\item For $(a,c,d) \in Y_{I,J} \cap (\bZ \cap [-H,H])^3 \setminus Z$ there exists $j \le \cJ$ for which
\[
g(a,c,d) = g_j(a,c,d) = 0.
\]
\end{enumerate}

Next, we let $G(a,c,d) \in \bZ[a,c,d]$ be coprime to $g$, and count solutions to
\begin{equation} \label{TwoEq}
g(a,c,d) = G(a,c,d) = 0.
\end{equation}
If $\deg_d(G) = 0$ then let $F(a,c) = G(a,c,d)$. Otherwise, let $F(a,c)$ be the resultant of $g$ and $G$ in the variable $d$. By \cite[Ch. 3, \S 6, Proposition 3]{CLO2015}, applied with $k$ as the fraction field of $\bZ[a,c]$, this is a non-zero element of $\bZ[a,c]$. By \cite[Ch. 3, \S 6, Proposition 5]{CLO2015}, we have $F(a,c) = 0$ for any solution $(a,c,d)$ to \eqref{TwoEq}.

Observe that $F(a,c) = 0$ if and only if we have $\cF(a,c) = 0$ for some irreducible factor $\cF(a,c) \in \bQ[a,c]$ of $F(a,c)$. So let $\cF(a,c) \in \bQ[a,c]$ be an irreducible factor of $F(a,c)$. If $\cF(a,c)$ is nonlinear, then Bombieri--Pila \cite[Corollary 1]{Die2013} gives
\[
\# \{ (a,c) \in (\bZ \cap [-H,H])^2: \cF(a,c) = 0 \} \ll H^{\frac12 + \eps}.
\]
Then $d$ is determined by $g(a,c,d) = 0$ in at most three ways, so the number of solutions $(a,c,d)$ counted in this case is $O(H^{\frac12+\eps})$.

Suppose instead that $\cF(a,c)$ is linear. Now
\[
\alp a + \bet c + \gam = 0,
\]
for some $(\alp, \bet, \gam) \in (\bQ^2 \setminus \{ (0,0) \}) \times \bQ$. If $\bet \ne 0$ then substitute $c = -\bet^{-1} (\alp a + \gam)$ into $g(a,c,d) = 0$, giving
\[
c_ 3 d^3 + P_2(a)d^2 + P_1(a) d + P_0(a) = 0,
\]
where 
\[
P_i(a) = c_i(a, -\bet^{-1} (\alp a + \gam)) \in \bQ[a] \qquad (i=0,1,2).
\]
Factorise the left hand side over $\bQ$, and let $\cP(a,d) \in \bQ[a,d]$ be an irreducible factor. Note that $\cP(a,d)$ is nonlinear, for if it were linear then
\[
\cP(a,d) = \cF(a,c) = 0
\]
would define a rational linear subvariety of $Y_{I,J}$, of dimension greater than or equal to 1, violating Lemma \ref{NoLines}. Now Bombieri--Pila yields
\[
\# \{ (a,d) \in (\bZ \cap [-H,H])^2: \cP(a,d) = 0 \} \ll H^{\frac12 + \eps}.
\]
If $\bet = 0$ then substitute $a = -\gam/\alp$ into $g(a,c,d) = 0$ and apply essentially the same reasoning.

In both cases, the number of integer solutions $(a,c,d) \in [-H,H]^3$ to \eqref{TwoEq} is $O(H^{\frac12 + \eps})$. We conclude that
\[
|Y_{I,J} \cap (\bZ \cap [-H,H])^3| \ll \cJ H^{\frac12+\eps} + H^{\frac2{\sqrt 6} + \eps} \ll H^{\frac1{\sqrt6} + \frac12 + 2\eps}.
\]
This concludes the proof of Theorem \ref{thm3}. Theorems \ref{thm1}, \ref{thm2} and \ref{thm3} imply Theorem \ref{nonS4}.

\section{Lower bounds}
\label{LowerBounds}

\subsection{Construction for $V_4$}
\label{lower}

Consider
\[
f(X) = X^4 + bX^2 + t^2,
\]
where $b, t \in \bN$ with 
\[
b \equiv 0 \mmod 4, \qquad t \equiv 1 \mmod 4
\]
and
\[
\frac12 H \le b \le H, \qquad t \le \sqrt H.
\]
Observe that the cubic resolvent
\[
r(X) = X^3 - bX^2 - 4t^2 X + 4bt^2 = (X-b)(X-2t)(X+2t)
\]
splits into linear factors over the rationals. If we can show that $f$ is irreducible over $\bQ$, then it will follow from Theorem \ref{KW} that $G_f \simeq V_4$.

Plainly $f(x) > 0$ whenever $x \in \bR$, so $f(X)$ has no rational roots, and therefore no linear factors. Suppose for a contradiction that $f(X)$ is reducible. Then by Gauss's lemma
\[
f(X) = (X^2 + pX + q)(X^2 + rX + s),
\]
for some $p,q,r,s \in \bZ$. Considering the $X^3$ coefficient of $f$ gives $r = -p$.

We begin with the case $p \ne 0$. Then considering the $X$ coefficient of $f$ gives $s = q$. Now
\[
X^4 + bX^2 + t^2 = (X^2 + pX + q)(X^2 - pX + q) = X^4 + (2q-p^2)X^2 + q^2,
\]
so $q = \pm t$ and $2q - b = p^2 \ge 0$. This is impossible, since
\[
b \ge H/2 > 2 \sqrt H \ge 2t = |2q|.
\]

It remains to consider the case $p = 0$. Now
\[
X^4 + bX^2 + t^2 = (X^2 + q)(X^2 + s),
\]
so
\[
q+s = b, \qquad qs = t^2.
\]
In particular $b^2 - 4t^2$ is a square, which is impossible because
\[
b^2 - 4t^2 \equiv 12 \mmod 16.
\]

Both cases led to a contradiction. Therefore $f$ is irreducible, and we conclude that $G_f \simeq V_4$. Our construction shows that $N_{V_4} \gg H^{3/2}$.

\subsection{Construction for $A_4$}
\label{sonne}

We use a construction motivated by \cite[Theorem 1.1]{NV1983}. Consider
the family of quartic polynomials
\[
  f(X) = f_{u,v}(X) = X^4+18v^2X^2+8uvX+u^2.
\]
Observe that $f(X)$ is irreducible in $\bZ[X,u,v]$, as $f_{1,0}(X) = X^4+1$ is irreducible in $\bZ[X]$. Next, consider the cubic resolvent of $f$, given by
\[
  r(X) = r_{u,v}(X) = X^3-18v^2 X^2-4u^2X+8u^2 v^2.
\]
This is also irreducible in $\bZ[X,u,v]$, as $r_{1,1}(X) = X^3-18X^2-4X+8$ is irreducible  in $\bZ[X]$. Hence, by Hilbert's irreducibility theorem \cite[Theorem 2.5]{Coh1981}, almost all specialisations $u,v \in \bN$ with $u,v \le \sqrt H/5$ give rise to an irreducible
$f(X) \in \zet[X]$ whose cubic resolvent is also irreducible. Finally, a short calculation reveals that
\[
  \operatorname{disc}(f(X))=(16(27uv^4 +u^3))^2,
\]
so these polynomials have Galois group $G_f \simeq A_4$. They are distinct, so $N_{A_4}(H) \gg H$.

\appendix

\section{Code}
\label{Code}

We used the C programming language to compute the values of $N_{G,4}(150)$ provided in the introduction, using GCC 4.2.1 as a compiler. The code is given below.

\begin{Verbatim} [fontsize=\scriptsize]

#include <stdio.h>
#include <math.h>
#include <stdlib.h>
#define RANGE 150 /* be careful of space for divisors */

char irred[2*RANGE+1][2*RANGE+1][2*RANGE+1][2*RANGE+1];
int divisors[RANGE*RANGE*RANGE+5*RANGE*RANGE+1][100];
/* again be careful of space for divisors */

/* irred entry is 1 if X^4+a*X^3+b*X^2+c*X+d irreducible otherwise 0
** divisors[i][0]: number of divisors of i
** divisors[i][j]: j-th divisor of i
** int needs to be at least 32 bit, long at least 64 bit */

void mark(int a, int b, int c, int d) {
  irred[a+RANGE][b+RANGE][c+RANGE][d+RANGE]=0;
}

void generate_irred() {
/* generate table of all irreducible monic quartic polynomials of height \le H
** first all having constant term zero
** next those splitting as (X+a)(X^3+b*X^2+c*X+d), where |a|, |d| \le H, |b|,|c| \le 2H
** finally those splitting as (X^2+a*X+b)(X^2+c*X+d), where |b|, |d| \le H, |a|, |c| \le 2H */
  int a, b, c, d;
  for (a=-RANGE; a<=RANGE; a++)
    for (b=-RANGE; b<=RANGE; b++)
      for (c=-RANGE; c<=RANGE; c++)
        for (d=-RANGE; d<=RANGE; d++)
          irred[a+RANGE][b+RANGE][c+RANGE][d+RANGE]=d!=0;
  for (a=-RANGE; a<=RANGE; a++)
    for (b=-2*RANGE; b<=2*RANGE; b++)
      for (c=-2*RANGE; c<=2*RANGE; c++)
        for (d=-RANGE; d<=RANGE; d++)
          if (abs(a+b)<=RANGE && abs(a*b+c)<=RANGE && abs(a*c+d)<=RANGE && abs(a*d)<=RANGE)
            mark(a+b, a*b+c, a*c+d, a*d);
   for (a=-2*RANGE; a<=2*RANGE; a++) 
    for (b=-RANGE; b<=RANGE; b++)
      for (c=-2*RANGE; c<=2*RANGE; c++)
        for (d=-RANGE; d<=RANGE; d++)
          if (abs(a+c)<=RANGE && abs(b+d+a*c)<=RANGE && abs(a*d+b*c)<=RANGE && abs(b*d)<=RANGE)
            mark(a+c, b+d+a*c, a*d+b*c, b*d); 
}

void generate_divisors() {
/* generate divisor list, see above; the range covers all potential divisors of the 
** constant term of the cubic resolvent of a monic quartic polynomial of height \le H */
  int i, j, n;
  for (i=1; i<=RANGE*RANGE*RANGE+5*RANGE*RANGE; i++) {
    for (n=0, j=1; j<=2*RANGE; j++) {
      if (i%j==0) 
        divisors[i][++n]=j;
    }
    divisors[i][0]=n;
  }
}

int is_square(long x) {
/* returns 1 if x is a square, 0 otherwise */
  long double y;
  y=ceil(sqrt(x));
  return y*y==x;
}

long discr(int a, int b, int c, int d) {
/* returns the discriminant of X^4+a*X^3+b*X^2+c*X+d */
  long a2, a3, a4, b2, b3, b4, c2, c3, c4, d2, d3;
  a2=a*a; b2=b*b; c2=c*c; d2=d*d;
  a3=a*a2; a4=a2*a2; b3=b*b2; b4=b2*b2; c3=c*c2; c4=c2*c2; d3=d*d2;
  return a2*b2*c2-4*b3*c2-4*a3*c3+18*a*b*c3-27*c4-4*a2*b3*d+16*b4*d+18*a3*b*c*d \
 -80*a*b2*c*d-6*a2*c2*d+144*b*c2*d-27*a4*d2+144*a2*b*d2-128*b2*d2-192*a*c*d2+256*d3;
}

int resolvent_reducible(int a, int b, int c, int d, int *root) {
/* returns 1 if the cubic resolvent X^3-b*X^2+(ac-4d)X-(a^2d-4bd+c^2) of X^4+a*X^3+b*X^2+c*X+d 
** is reducible, in which case root will be an integer root of the resolvent;
** otherwise return 0, root undefined. For C4 and D4 the root is unique */
  int i, x, y, q, r, ra;
  r=a*a*d-4*b*d+c*c;
  if (r==0) {
    *root=0; return 1;
  }
  q=a*c-4*d;
  ra=abs(r);
  for (i=1; i<=divisors[ra][0]; i++) {
    x=divisors[ra][i];
    if (x*x*x-b*x*x+q*x-r==0) {
      *root=x; return 1;
    }
    y=-x;
    if (y*y*y-b*y*y+q*y-r==0) {
      *root=y; return 1;
    }
  }
  return 0;
}

void loop_over_b_c_d(long *s4, long *a4, long *d4, long *c4, long *v4, long *red, int a, int f) {
  long disc;
  int b, c, d, res_red, root;
  for (b=-RANGE; b<=RANGE; b++)
    for (c=-RANGE; c<=RANGE; c++)
      for (d=-RANGE; d<=RANGE; d++)
        if (irred[a+RANGE][b+RANGE][c+RANGE][d+RANGE]) {
          res_red=resolvent_reducible(a,b,c,d,&root);
          disc=discr(a,b,c,d);
          if (is_square(disc))
            res_red ? (*v4+=f) : (*a4+=f);
          else {
            if (res_red)
              is_square((root*root-4*d)*disc) && is_square((a*a-4*(b-root))*disc)?(*c4+=f):(*d4+=f);
            else
              *s4+=f;
          }
        }
        else
          *red+=f;
}

int main() {
/* Following the criteria in our paper, loop a,b,c,d over the height RANGE, each time compute 
** the Galois group of X^4+a*X^3+b*X^2+c*X+d and print the resulting statistics */
  long s4=0, a4=0, d4=0, c4=0, v4=0, red=0;
  int a;
  generate_irred();
  generate_divisors();
  loop_over_b_c_d(&s4, &a4, &d4, &c4, &v4, &red, 0, 1);
  for (a=1; a<=RANGE; a++)
    loop_over_b_c_d(&s4, &a4, &d4, &c4, &v4, &red, a, 2);
  printf("Number of \033[1mreducible\033[22m polynomials of height at most %d: %ld\n", RANGE, red);
  printf("Number of \033[1mS4\033[22m polynomials of height at most %d:        %ld\n", RANGE, s4);
  printf("Number of \033[1mA4\033[22m polynomials of height at most %d:        %ld\n", RANGE, a4);
  printf("Number of \033[1mD4\033[22m polynomials of height at most %d:        %ld\n", RANGE, d4);
  printf("Number of \033[1mV4\033[22m polynomials of height at most %d:        %ld\n", RANGE, v4);
  printf("Number of \033[1mC4\033[22m polynomials of height at most %d:        %ld\n", RANGE, c4);
}
\end{Verbatim}

\bigskip

Below is the code to compute $N_{A_3,3}(2000)$.

\begin{Verbatim} [fontsize=\scriptsize]

#include <stdio.h>
#include <math.h>
#include <stdlib.h>
#define RANGE 2000

char irred[2*RANGE+1][2*RANGE+1][2*RANGE+1];

/* 1 if X^3+a*X^2+b*X+c irreducible otherwise 0 */

void mark(int a, int b, int c) {
  irred[a+RANGE][b+RANGE][c+RANGE]=0;
}

void generate_irred() {
/* generate table of all irreducible monic cubic polynomials of height <= H
** first all having constant term zero
** next those splitting as (X+a)(X^2+b*X+c), where |a| <=H, |b|<=2H, |c|<=H */
  int a, b, c;
  for (a=-RANGE; a<=RANGE; a++)
    for (b=-RANGE; b<=RANGE; b++)
      for (c=-RANGE; c<=RANGE; c++)
        irred[a+RANGE][b+RANGE][c+RANGE]=c!=0;
  for (a=-RANGE; a<=RANGE; a++)
    for (b=-2*RANGE; b<=2*RANGE; b++)
      for (c=-RANGE; c<=RANGE; c++)
          if (abs(a+b)<=RANGE && abs(a*b+c)<=RANGE && abs(a*c)<=RANGE)
            mark(a+b, a*b+c, a*c);
}

int is_square(long x) {
/* returns 1 if x is a square, 0 otherwise */
  long double y;
  y=ceil(sqrt(x));
  return y*y==x;
}

long discr(long a, long b, long c) {
/* returns the discriminant of X^3+a*X^2+b*X+c */
  return (b*b-4*a*c)*(a*a-4*b)+c*(2*a*b-27*c);
}

int main() {
  long s3=0, a3=0, red=0;
  generate_irred();
  for (int a=-RANGE; a<=RANGE; a++)
    for (int b=-RANGE; b<=RANGE; b++)
      for (int c=-RANGE; c<=RANGE; c++)
        if (irred[a+RANGE][b+RANGE][c+RANGE])
          if (is_square(discr(a,b,c)))
            a3++;
          else
            s3++;
          else
            red++;
  printf("Number of \033[1mreducible\033[22m polynomials of height at most %d: %ld\n", RANGE, red);
  printf("Number of \033[1mS3\033[22m polynomials of height at most %d:        %ld\n", RANGE, s3);
  printf("Number of \033[1mA3\033[22m polynomials of height at most %d:        %ld\n", RANGE, a3);
}
\end{Verbatim}

\section{Counting reducible polynomials}
\label{Red}

In this appendix we trace through Chela's proof \cite{Che1963} to verify the error term in \eqref{ReducibleCount}. The outcome should be unsurprising if one considers Dubickas's corresponding error term in the non-monic setting \cite{Dub2014}. The implicit constants are allowed to depend on the degree $n$. We may assume, for ease of notation, that $H$ is an integer. It may help the reader to know that
\[
c_n = 2^n( \zeta(n-1)-1 ) + 2^{n-1} + 2k_n,
\]
where $\zeta(\cdot)$ denotes the Riemann zeta function and $k_n$ denotes the Euclidean volume of the region $\cR \subset \bR^{n-1}$ defined by
\[
|x_i| \le 1 \quad (1 \le i \le n-1), \qquad \Bigl| \sum_{i=1}^{n-1} x_i \Bigr| \le 1.
\]

As Chela explains from the outset, van der Waerden had already shown that the number of $f$ given by \eqref{fdef} having a factor of degree $k \in [2,n/2]$ with $|a_i| \le H$ for all $i$ is $O(H^{n-2} \log H)$. Thus, we need only to count polynomials with a linear factor $X+v$, so suppose that there are $T(v)$ of these. 

To deal with the issue of over-counting, Chela bounds the number of polynomials with at least two (not necessarily distinct) linear factors. Chela's reasoning is that these polynomials have a quadratic factor, and if $n \ge 4$ then this reveals that there are $O(H^{n-2})$ such polynomials. In the case $n = 3$ this reasoning breaks down, but a standard mean value estimate for the arithmetic function 
\[
\tau_3(m) = \displaystyle \sum_{d_1d_2d_3 = m}1
\]
procures the bound $O(H(\log H)^2)$, and this is satisfactory.

Following \cite{Che1963} up to Eq. (17) therein, we see that the effective error version
\[
\sum_{|v|>1} T(v) = 2 \sum_{v=2}^{H-1} (2H/v)^{n-1} + O(H^{n-2})
= 2^n  (\zeta(n-1) - 1) H^{n-1} + O(H^{n-2})
\]
holds. As
\[
T(0) = (2H)^{n-1} + O(H^{n-2})
\]
and $T(1) = T(-1)$, it remains to show that
\begin{equation*}
T(-1) = k_n H^{n-1} + O(H^{n-2} (\log H)^2).
\end{equation*}
To this end, since if $X-1$ divides $f(X)$ then
\[
0 = f(1) = 1 + a_1 + \cdots + a_n,
\]
the final task is to count polynomials with
\[
a_1 + \cdots + a_n = -1.
\]

For $h \in \bZ$ and $N \in \bZ_{\ge 2}$, write $L(N,h)$ for the number of vectors $(a_1,\ldots,a_n) \in \bZ^n$ such that
\[
\max_i |a_i| \le N, \qquad a_1 + \cdots + a_n = h.
\]
By symmetry $L(-1) = L(1)$, so it suffices to prove that
\[
L(H,1) = k_nH^{n-1} + O(H^{n-2}).
\]
From \cite[Eq. (27)]{Che1963}, we have
\[
L(H-1, 0) \le L(H,1) \le L(H+1,0).
\]
It therefore remains to show that
\begin{equation} \label{LN0}
L(N,0) = k_n N^{n-1} + O(N^{n-2}).
\end{equation}
The quantity $L(N,0)$ equivalently counts lattice points $(a_1, \ldots, a_{n-1})$ in the region $N \cR$, so by standard geometry of numbers \cite[Lemma 1]{Sch1995} we obtain \eqref{LN0}.

\section{Binary quartic forms with given invariants}
\label{Aux}

In this appendix, we prove the following result related to Lemma \ref{NoLines}. In words, it asserts that given $I,J \in \bC$ for which the discriminant $4I^3 - J^2$ is non-zero, the space of binary quartic forms with these invariants contains no complex lines. A rational line on the variety induces a complex line on the variety, by equating coefficients, so a consequence is that there are no rational lines. See \cite{BS2015} for further information about the invariants $I$ and $J$ of a binary quartic form $aX^4 + bX^3 Y + cX^2 Y^2 + dXY^3 + eY^4$.

\begin{thm} Let $I, J \in \bC$ with $4I^3 - J^2 \ne 0$. Then the affine subvariety $Z = Z_{I,J}$ of $\bA_\bC^5$ defined by
\[
I = 12ae - 3bd + c^2, \qquad 
J =  72ace + 9bcd - 27ad^2 - 27b^2e - 2c^3
\]
contains no lines.
\end{thm}

A line takes the form
\[
\cL = \{ (\alp, \bet, \gam, \del, \epsilon) + t(A,B,C,D,E): t \in \bC \},
\]
for some $(\alp, \bet, \gam, \del, \epsilon) \in \bC^5$ and some $(A,B,C,D,E) \in \bC^5 \setminus \{ \bzero \}$. There are five types of line to consider:
\begin{enumerate}[I.]
\item $\cL = \{ (0, \bet, \gam, \del, \epsilon) + t(1,B,C,D,E): t \in \bC \}$;
\item $\cL = \{ (\alp, 0, \gam, \del, \epsilon) + t(0,1,C,D,E): t \in \bC \}$;
\item $\cL = \{ (\alp, \bet, 0, \del, \epsilon) + t(0,0,1,D,E): t \in \bC \}$;
\item $\cL = \{ (\alp, \bet, \gam, 0, \epsilon) + t(0,0,0,1,E): t \in \bC \}$;
\item $\cL = \{ (\alp, \bet, \gam, \del, 0) + t(0,0,0,0,1): t \in \bC \}$.
\end{enumerate}

\bigskip

In each case, we expanded the expressions for $I$ and $J$ as polynomials in $t$. Equating coefficients then provided seven equations, and we used the software \emph{Mathematica} \cite{Mathematica} to obtain $4I^3 - J^2 = 0$ by elimination of variables. For example, in Case I, the code is as follows.

\begin{Verbatim} [fontsize=\scriptsize]

a = t; b = \[Beta] + t p; c = \[Gamma] + t q; 
d = \[Delta] + t r;  e = \[Epsilon] + t s;
Collect[12 a e - 3 b d + c^2, t]
Collect[72 a c e + 9 b c d - 27 a d^2 - 27 b^2 e - 2 c^3, t]
Eliminate[q^2 - 3 p r + 12 s == 0 && \[Gamma]^2 - 3 \[Beta] \[Delta] == k && 
-3 r \[Beta] + 2 q \[Gamma] - 3 p \[Delta] + 12 \[Epsilon] == 0 && 
-2 q^3 + 9 p q r - 27 r^2 - 27 p^2 s + 72 q s ==  0 && 
-2 \[Gamma]^3 + 9 \[Beta] \[Gamma] \[Delta] - 27 \[Beta]^2 \[Epsilon] == j && 
 9 q r \[Beta] - 54 p s \[Beta] - 6 q^2 \[Gamma] + 9 p r \[Gamma] + 72 s \[Gamma] 
+ 9 p q \[Delta] - 54 r \[Delta] - 27 p^2 \[Epsilon] + 72 q \[Epsilon] == 0 && 
-27 s \[Beta]^2 + 9 r \[Beta] \[Gamma] - 6 q \[Gamma]^2 + 9 q \[Beta] \[Delta] 
+ 9 p \[Gamma] \[Delta] - 27 \[Delta]^2 -  54 p \[Beta] \[Epsilon] + 72 \[Gamma] \[Epsilon] 
== 0, {\[Beta], \[Gamma], \[Delta], \[Epsilon], p, q, r, s}]

\end{Verbatim}

\noindent Cases II, IV, and V also lead to $4I^3 - J^2 = 0$, whilst Case III can never occur. We have deduced $4I^3 - J^2 = 0$ from the existence of a complex line, completing the proof of the theorem.

\providecommand{\bysame}{\leavevmode\hbox to3em{\hrulefill}\thinspace}

\end{document}